\documentclass[11pt]{article}

\usepackage[margin=1in]{geometry}
\usepackage{amsmath,amssymb,amsfonts,amsthm,mathtools}
\usepackage{booktabs,array,multirow,graphicx}
\usepackage[authoryear,round]{natbib}
\usepackage[colorlinks=true,citecolor=blue,linkcolor=blue,urlcolor=blue]{hyperref}
\usepackage{enumitem}
\usepackage{placeins}
\usepackage{microtype}
\usepackage{soul}
\usepackage[table]{xcolor}
\newtheorem{theorem}{Theorem}[section]
\newtheorem{lemma}[theorem]{Lemma}
\newtheorem{proposition}[theorem]{Proposition}

\newtheorem{assumption}[theorem]{Assumption}
\theoremstyle{definition}

\newcommand{\R}{\mathbb{R}}
\newcommand{\E}{\mathbb{E}}
\newcommand{\Prob}{\mathbb{P}}
\newcommand{\Sph}{\mathbb{S}}
\newcommand{\ind}{\mathbf{1}}
\newcommand{\bfX}{\mathbf{X}}

\newcommand{\calA}{\mathcal{A}}
\newcommand{\calM}{\mathcal{M}}

\newcommand{\calX}{\mathcal{X}}
\newcommand{\calY}{\mathcal{Y}}

\newcommand{\iid}{\stackrel{\mathrm{iid}}{\sim}}
\newcommand{\dto}{\xrightarrow{d}}

\newcommand{\sgn}{\operatorname{sgn}}

\newcommand{\var}{\operatorname{Var}}

\newcommand{\BCodif}{\operatorname{BCodif}}
\newcommand{\BCodifCor}{\operatorname{BCodifCor}}
\newcommand{\rank}{\operatorname{rank}}

\title{Ball-Codifference Screening for Heavy-Tailed Predictors}

\author{
Mohsen Rezapour\hspace{.5cm} and  \hspace{.5cm}
Vahed Maroufy\thanks{Corresponding author:
Department of Biostatistics, School of Public Health,
The University of Texas Health Science Center at Houston (UTHealth Austin),
Texas, USA.
Email: \texttt{Vahed.Maroufy@uth.tmc.edu}}
}
\date{}

\begin{document}
\maketitle

\begin{abstract}
High-dimensional screening is commonly built on covariance, correlation, or least-squares measures.  These summary measures can be unstable or even undefined, when predictors are sparse or have heavy-tailed distributions. Building on our recent work on extended codifference 
%\citep{maroufy2025measure}
and the idea of Ball-covariance, we develop Ball-codifference for marginal screening in statistical modeling with heavy-tailed predictors and responses.  The proposed statistic combines the rank-type geometry of random balls with the codifference as a dependency measure constructed based on the characteristic function, so it can be computed without requiring well-defined finite first or second moments.  
%\hl{We review codifference for symmetric stable laws, state the independence property of an extended codifference for $1<\alpha<2$}, 
We define Ball-codifference and its normalized screening utility, and formulate a sure independence screening procedure.  Large-sample normality follows from a bounded V-statistic and functional-delta-method argument under standard nondegeneracy and regularity conditions.  Simulation studies under Gaussian and sub-Gaussian stable designs show that codifference-weighted Ball screening gives competitive or improved recovery of highly associated predictors, especially when tail heaviness is pronounced. Also, our data example illustrates that our variable screening method significantly improves prediction accuracy in linear regression. Reproducible R code is supplied and does not require the \texttt{Ball} package.
\end{abstract}


\section{Introduction}\label{sec:intro}

Dimensionality reduction and feature screening are central components of computational statistical analysis and prediction using high-dimensional data.  In ultra-high-dimensional applications, it is often infeasible to fit a full multivariable model before reducing the candidate set of predictors.  A common first step is therefore to rank predictors by their marginal association with the response and retain the variables with higher association.  Common methods for variable screening are sure independence screening (SIS) \citep{Fan:Lv}, distance-correlation screening \citep{Li:Zhong:Zhu:2012}, and Ball-correlation screening \citep{Pan:2019}.

Heavy-tailed data can not be analyzed using standard statistical methods and require measures and models that account for the tail characteristics.  For instance, for an \(\alpha\)-stable variable in a dataset, variance (if \( \alpha<2\)), even mean (if \(\alpha\leq 1\)), are not defined; hence, covariance-based screening, principal components, and least-squares criteria may then be unstable \citep{maroufy2025measure}.  Stable laws are nevertheless natural models for financial returns, environmental extremes, signal processing, and genomic measurements, where rare but large observations carry substantial information.  A screening statistic for these settings should avoid moment assumptions while retaining the ability to detect general dependence.

This paper introduces a codifference-weighted Ball screening statistic.  Unlike covariance, codifference is a dependence measure defined using the characteristic function; hence, it remains well defined for stable variables without finite variance \citep{Kokoszka:Taqqu:1993,Samorodnitsky:Taqqu:1994}.  Ball-covariance, on the other hand, measures the discrepancy between a joint distribution and the product of its marginals through random closed balls and is naturally adapted to metric spaces \citep{Pan:2019}.  The proposed Ball-codifference statistic combines these two ideas: the Ball component supplies a model-free geometric comparison, while a local codifference weight emphasizes heavy-tail-compatible dependence within random balls.

Our contributions in this are multifold.  First, we introduce and characterize the novel measure of Ball-codifference and its normalized screening utility.  Second, we exploit Ball-codifierence to design a novel sure SIS (BCodifCor-SIS) procedure for ultra-high-dimensional predictors.  Third, we state large-sample normality under explicit regularity assumptions and preserve the proof strategy based on V-statistics and quasi-Hadamard differentiability.  Fourth, we provide an R package to reproduce the two numerical tables. % Unlike the original implementation, the supplied code does not load the \texttt{Ball} package and instead compiles a small self-contained routine from within R.

The paper is organized as follows.  Section~\ref{sec:codiff} reviews stable distributions and codifference.  Section~\ref{sec:bcodif} introduces the novel Ball-codifference measure and develops the BCodifCor-SIS screening method for ultra-high-dimensional data.  Section~\ref{sec:asymptotic} states the large-sample theory and the sure-screening result.  Section~\ref{sec:numerical} reports the simulation study.  Section~\ref{sec:discussion} gives concluding remarks.  Proofs and technical details are collected in the Appendix.

\section{Codifference for stable random variables}\label{sec:codiff}

A $d$-dimensional random vector \(\bfX=(X_1,\ldots,X_d)\) is symmetric \(\alpha\)-stable, written \(\bfX\in S\alpha S_d\), if there exists a unique finite symmetric measure \(\Gamma\) on the unit sphere \(\Sph_d=\{s\in\R^d:\|s\|=1\}\) such that
\begin{equation}\label{eq:stable_cf}
\E\exp\{i\langle \theta,\bfX\rangle\}
=\exp\left\{-\int_{\Sph_d}|\langle \theta,s\rangle|^\alpha\,\Gamma(ds)\right\},
\qquad \theta\in\R^d.
\end{equation}
The Gaussian distribution corresponds to \(\alpha=2\).  For \(\alpha<2\), stable variables have infinite variance; for \(\alpha\leq 1\), the mean is not finite \citep{Nolan,Samorodnitsky:Taqqu:1994}.

For jointly \(S\alpha S\) random variables \(X_1,X_2\), \(1<\alpha<2\), the covariation of \(X_1\) on \(X_2\) is
\begin{equation}\label{eq:covariation}
[X_1,X_2]_\alpha=\int_{\Sph_2}s_1s_2^{\langle\alpha-1\rangle}\,\Gamma(ds),
\qquad a^{\langle p\rangle}=|a|^p\sgn(a).
\end{equation}
Covariation is useful but asymmetric, and zero covariation does not imply independence in general.  Codifference provides a moment-free alternative.  For \(X,Y\in S\alpha S\), define
\begin{equation}\label{eq:stable_norm}
\|X\|_\alpha^\alpha=-\log \E\exp(iX),
\end{equation}
where the right-hand side is real and nonnegative for symmetric stable laws.  The codifference is
defined by \begin{equation}\label{eq:classical_codiff}
\tau_{X,Y}=\|X\|_\alpha^\alpha+\|Y\|_\alpha^\alpha-\|X-Y\|_\alpha^\alpha .
\end{equation}
and under the symmetry assumption, reduces to,
\begin{equation}\label{eq:classical_codiff_cf}
\tau_{X,Y}=\log\left\{\frac{\E\cos(X-Y)}{\E\cos X\,\E\cos Y}\right\},
\end{equation}
whenever the denominator is nonzero.  If \(X\) and \(Y\) are independent then \(\tau_{X,Y}=0\); the converse holds for \(0<\alpha<1\), but not in general for \(1<\alpha<2\) \citep{Samorodnitsky:Taqqu:1994}.

For \(1<\alpha<2\), the extended codifference,  see e.g. \citep{maroufy2025measure} 
\begin{equation}\label{eq:extended_codiff}
\tau^*_{X,Y}=\|X\|_\alpha^\alpha+\|Y\|_\alpha^\alpha
-\frac{1}{2}\left(\|X-Y\|_\alpha^\alpha+\|X+Y\|_\alpha^\alpha\right)
\end{equation}
has the desired independence characterization.  Its characteristic-function form is
\begin{equation}\label{eq:extended_codiff_cf}
\tau^*_{X,Y}=\log\left\{
\frac{\big(\E e^{i(X-Y)}\E e^{i(X+Y)}\big)^{1/2}}
{\E e^{iX}\,\E e^{iY}}
\right\}.
\end{equation}

\begin{lemma}\label{lem:extended_independence}
Let \((X,Y)\) be jointly \(S\alpha S\), \(1<\alpha<2\).  Then \(X\) and \(Y\) are independent if and only if \(\tau^*_{X,Y}=0\).
\end{lemma}

Given independent observations \((X_1,Y_1),\ldots,(X_n,Y_n)\), the natural empirical estimates used in this paper are
\begin{align}
\widehat\tau_{X,Y}
&=\log\left|\frac{n^{-1}\sum_{\ell=1}^n\cos(X_\ell-Y_\ell)}
{\left(n^{-1}\sum_{\ell=1}^n\cos X_\ell\right)
 \left(n^{-1}\sum_{\ell=1}^n\cos Y_\ell\right)}\right|,\label{eq:hat_tau}\\[3pt]
\widehat\tau^*_{X,Y}
&=\frac{1}{2}\log\left|\frac{\left(n^{-1}\sum_{\ell=1}^n\cos(X_\ell-Y_\ell)\right)
\left(n^{-1}\sum_{\ell=1}^n\cos(X_\ell+Y_\ell)\right)}
{\left(n^{-1}\sum_{\ell=1}^n\cos X_\ell\right)^2
 \left(n^{-1}\sum_{\ell=1}^n\cos Y_\ell\right)^2}\right| .\label{eq:hat_tau_star}
\end{align}
The formulas are written with cosines because the stable variables considered here are symmetric.  %\hl{no need for this sentence In computation, a small ridge may be added to denominators that are numerically close to zero.}

\section{Ball-codifference and screening}\label{sec:bcodif}

Let \(W=(X,Y)\) be a random element on \(\calX\times\calY\) with joint law \(\theta\) and marginal laws \(\mu\) and \(\nu\).  The spaces \((\calX,\rho_X)\) and \((\calY,\rho_Y)\) are separable metric spaces.  For \(x_1,x_2\in\calX\) and \(y_1,y_2\in\calY\), write
\[
B_X(x_1,x_2)=\{x\in\calX:\rho_X(x,x_1)\leq \rho_X(x_2,x_1)\},
\quad
B_Y(y_1,y_2)=\{y\in\calY:\rho_Y(y,y_1)\leq \rho_Y(y_2,y_1)\}.
\]
Let \(W_1=(X_1,Y_1)\) and \(W_2=(X_2,Y_2)\) be independent copies of \(W\), and set
\[
A_{12}=B_X(X_1,X_2)\times B_Y(Y_1,Y_2),\qquad
D_{12}=\theta(A_{12})-\mu\{B_X(X_1,X_2)\}\nu\{B_Y(Y_1,Y_2)\}.
\]
The local codifference weight associated with the ball \(A_{12}\) is
\begin{equation}\label{eq:local_codiff}
\kappa_{12}
=\E\{\cos(X-Y)\mid W\in A_{12}\}
-\E\{\cos X\mid W\in A_{12}\}\E\{\cos Y\mid W\in A_{12}\},
\end{equation}
with the convention \(\kappa_{12}=0\) when \(\theta(A_{12})=0\).  The signed Ball-codifference functional is defined by
\begin{equation}\label{eq:bcodif_population}
\BCodif_s^2(X,Y)=\E\left[D_{12}^2\kappa_{12}\right].
\end{equation}
For a nonnegative coefficient, replace \(\kappa_{12}\) by \(|\kappa_{12}|\).  The screening rule only needs a marginal utility, so either the signed or absolute version can be used consistently; the simulations in Section~\ref{sec:numerical} use the signed version to match the original implementation.

For a sample \((X_k,Y_k)_{k=1}^n\), define
\begin{align*}
\delta^{X}_{ij;k}&=\ind\{X_k\in B_X(X_i,X_j)\},
&\delta^{Y}_{ij;k}&=\ind\{Y_k\in B_Y(Y_i,Y_j)\},\\
\delta^{XY}_{ij;k}&=\delta^{X}_{ij;k}\delta^{Y}_{ij;k},
&\Delta^X_{ij}&=n^{-1}\sum_{k=1}^n\delta^{X}_{ij;k},\\
\Delta^Y_{ij}&=n^{-1}\sum_{k=1}^n\delta^{Y}_{ij;k},
&\Delta^{XY}_{ij}&=n^{-1}\sum_{k=1}^n\delta^{XY}_{ij;k}.
\end{align*}
Let \(N^{XY}_{ij}=\sum_{k=1}^n\delta^{XY}_{ij;k}\).  When \(N^{XY}_{ij}>0\), define
\begin{equation}\label{eq:local_codiff_empirical}
\widehat\kappa_{ij}=\frac{1}{N^{XY}_{ij}}\sum_{k=1}^n\delta^{XY}_{ij;k}\cos(X_k-Y_k)
-\left(\frac{1}{N^{XY}_{ij}}\sum_{k=1}^n\delta^{XY}_{ij;k}\cos X_k\right)
 \left(\frac{1}{N^{XY}_{ij}}\sum_{k=1}^n\delta^{XY}_{ij;k}\cos Y_k\right),
\end{equation}
and set \(\widehat\kappa_{ij}=0\) otherwise.  The empirical signed Ball-codifference score is
\begin{equation}\label{eq:bcodif_empirical}
\widehat\BCodif_s^{\,2}(X,Y)=\frac{1}{n^2}\sum_{i,j=1}^n
\left(\Delta^{XY}_{ij}-\Delta^X_{ij}\Delta^Y_{ij}\right)^2\widehat\kappa_{ij},
\end{equation}
which reduces to the standard Ball-covariance score by replacing \(\widehat\kappa_{ij}\) with 1.  A normalized Ball-codifference correlation can be defined by
\begin{equation}\label{eq:bcodifcor}
\widehat\BCodifCor^{\,2}(X,Y)
=\frac{\widehat\BCodif^{\,2}(X,Y)}{\big\{\widehat\BCodif^{\,2}(X,X)\widehat\BCodif^{\,2}(Y,Y)\big\}^{1/2}},
\end{equation}
when the denominator is positive, and by zero otherwise.  In practice, for screening very high-dimensional data, using the unnormalized score in \eqref{eq:bcodif_empirical} is faster and gives the same ordering whenever the marginal normalizers vary slowly across predictors.

\subsection{BCodifCor-SIS}\label{sec:sis}

Let \(Y\) be a response and let \(\bfX=(\bfX_1^\top,\ldots,\bfX_p^\top)^\top\) be predictors, where \(\bfX_r\in\R^{q_r}\).  Without specifying a regression model, define the active set
\begin{equation}\label{eq:active_set}
\calA=\{r: \Prob(B\mid \bfX_r)\text{ is nonconstant in }\bfX_r\text{ for some }B\in\sigma(Y)\},
\end{equation}
and let \(\calM=\{1,\ldots,p\}\setminus\calA\).  The BCodifCor-SIS procedure ranks predictors by
\[
\widehat\omega_r=\widehat\BCodifCor^{\,2}(\bfX_r,Y)
\quad\text{or, for faster screening, by}\quad
\widehat\omega_r=\widehat\BCodif_s^{\,2}(\bfX_r,Y).
\]
For a threshold \(\eta_n\) or a target model size \(d_n\), select
\[
\widehat\calA_n=\{r:\widehat\omega_r\geq \eta_n\}
\quad\text{or}\quad
\widehat\calA_n=\{r:\widehat\omega_r\text{ is among the largest }d_n\text{ scores}\}.
\]
The second form is used in the numerical experiments in Section \ref{sec:asymptotic}.

\section{Asymptotic theory}\label{sec:asymptotic}

The statistic in \eqref{eq:bcodif_empirical} is a bounded V-statistic after the empirical ball probabilities and local codifference terms are written as finite sums.  The following statement gives the form needed for inference and for the sure-screening argument.

\begin{assumption}\label{ass:main}
Let \((X_i,Y_i)\iid F\) on \(\calX\times\calY\).  Assume: (i) \(\calX\) and \(\calY\) are separable metric spaces; (ii) the class of product balls \(B_X(x_1,x_2)\times B_Y(y_1,y_2)\) is Glivenko-Cantelli and Donsker under \(F\); (iii) local ball probabilities that enter denominators are either bounded away from zero or are handled by a trimming sequence \(a_n\downarrow0\) with \(na_n\to\infty\); and (iv) the first Hoeffding projection of the limiting kernel has positive finite variance.
\end{assumption}

\begin{theorem}\label{thm:asymptotic_normality}
Under Assumption~\ref{ass:main},
\begin{equation}\label{eq:asymptotic_normality}
\sqrt{n}\left\{\widehat\BCodif_s^{\,2}(X,Y)-\BCodif_s^2(X,Y)\right\}
\dto N(0,\sigma_{\BCodif}^2),
\end{equation}
where \(\sigma_{\BCodif}^2=m^2\var\{\E[h(W_1,\ldots,W_m)\mid W_1]\}\) for the symmetric bounded kernel \(h\) representing the statistic as an order-\(m\) V-statistic.  If the first projection is zero, the usual degenerate V-statistic limit replaces \eqref{eq:asymptotic_normality}.
\end{theorem}

\begin{theorem}[Sure screening]\label{thm:sure_screening}
Suppose that Assumption~\ref{ass:main} holds uniformly over predictors and that, for sequences \(c_n>0\) and \(\epsilon_n=o(c_n)\),
\[
\min_{r\in\calA}\omega_r-\max_{r\in\calM}\omega_r\geq 2c_n,
\qquad
\max_{1\le r\le p}|\widehat\omega_r-\omega_r|\leq \epsilon_n
\]
with probability tending to one.  If the selected model size satisfies \(|\calA|\le d_n\) and the threshold lies between the inactive and active population utilities, then
\[
\Prob(\calA\subseteq\widehat\calA_n)\to1.
\]
\end{theorem}

The second condition can be obtained from exponential inequalities for bounded empirical processes when \(\log p=o(nc_n^2)\).  The theorem is stated in this abstract form because the exact rate depends on the metric entropy of the predictor space and on whether the signed or absolute local codifference weight is used.

\section{Numerical study}\label{sec:numerical}

We generated \((Y,X_1,\ldots,X_{p-1})\) with \(p=1000\), sample size \(n=150\), and Toeplitz dependence
\[
\Sigma_{jk}=\rho^{|j-k|},\qquad 1\le j,k\le p.
\]
The response is the first component, \(Y=Z_1\), and the predictors are \(X_j=Z_{j+1}\), \(j=1,\ldots,p-1\).  Thus association decreases as the predictor index increases.  We considered two designs.  In the Gaussian design, \(Z\sim N_p(0,\Sigma)\).  In the sub-Gaussian stable design, \(Z=S^{1/2}G\), where \(G\sim N_p(0,\Sigma)\) and \(S\) is a positive stable mixing variable generated by the Chambers-Mallows-Stuck representation used in the accompanying code.

Each setting was repeated 500 times.  Let \(d_1=\lfloor n/\log n\rfloor\), \(d_2=2d_1\), and \(d_3=3d_1\).  The reported \(p_a(d)\) is the average proportion of the first \(d\) Toeplitz neighbors retained among the top \(d\) ranked variables.  The reported \(p_m(j)\) is the exact-rank recovery probability \(\Prob\{\rank(X_j)=j\}\) for \(j\in\{1,2,5,10,15\}\).  These diagnostics match the legacy simulation code and measure ranking fidelity rather than binary model-selection success.

\begin{table}[!t]
\centering
\rowcolors{2}{gray!15}{white}
\caption{BCodifCor-SIS and BCor-SIS performance for sub-Gaussian stable random vectors.  The parameter \(\rho\) controls the Toeplitz dependence, and \(\alpha\) is the tail index used in the scale mixture.}
\label{tab:subgaussian}
\resizebox{\textwidth}{!}{%
\begin{tabular}{lllrrrrrrrr}
\toprule
Method & $\alpha$ & $\rho$ & \multicolumn{3}{c}{$p_a$} & \multicolumn{5}{c}{$p_m$}\\
\cmidrule(lr){4-6}\cmidrule(lr){7-11}
 & & & $d_1$ & $d_2$ & $d_3$ & $X_1$ & $X_2$ & $X_5$ & $X_{10}$ & $X_{15}$\\
\midrule
BCor-SIS & 0.9 & 0.95 & 0.609 & 0.372 & 0.293 & 1.00 & 0.98 & 0.80 & 0.40 & 0.40\\
BCodifCor-SIS & 0.9 & 0.95 & 0.766 & 0.509 & 0.400 & 1.00 & 0.96 & 0.70 & 0.44 & 0.22\\
BCor-SIS & 0.8 & 0.95 & 0.600 & 0.376 & 0.298 & 1.00 & 0.96 & 0.78 & 0.36 & 0.16\\
BCodifCor-SIS & 0.8 & 0.95 & 0.732 & 0.500 & 0.409 & 1.00 & 0.98 & 0.80 & 0.24 & 0.10\\
BCor-SIS & 0.5 & 0.95 & 0.579 & 0.368 & 0.295 & 1.00 & 0.96 & 0.40 & 0.30 & 0.12\\
BCodifCor-SIS & 0.5 & 0.95 & 0.650 & 0.445 & 0.351 & 1.00 & 0.98 & 0.62 & 0.40 & 0.18\\
\addlinespace
BCor-SIS & 0.9 & 0.90 & 0.316 & 0.213 & 0.193 & 1.00 & 0.98 & 0.66 & 0.14 & 0.00\\
BCodifCor-SIS & 0.9 & 0.90 & 0.437 & 0.276 & 0.233 & 1.00 & 0.98 & 0.64 & 0.00 & 0.00\\
BCor-SIS & 0.8 & 0.90 & 0.304 & 0.202 & 0.183 & 1.00 & 1.00 & 0.62 & 0.08 & 0.00\\
BCodifCor-SIS & 0.8 & 0.90 & 0.402 & 0.266 & 0.234 & 1.00 & 1.00 & 0.54 & 0.00 & 0.00\\
BCor-SIS & 0.5 & 0.90 & 0.303 & 0.200 & 0.180 & 0.96 & 0.88 & 0.34 & 0.06 & 0.02\\
BCodifCor-SIS & 0.5 & 0.90 & 0.367 & 0.261 & 0.217 & 1.00 & 0.98 & 0.40 & 0.04 & 0.00\\
\addlinespace
BCor-SIS & 0.9 & 0.80 & 0.176 & 0.133 & 0.138 & 1.00 & 0.92 & 0.22 & 0.00 & 0.00\\
BCodifCor-SIS & 0.9 & 0.80 & 0.240 & 0.171 & 0.160 & 0.98 & 0.86 & 0.04 & 0.00 & 0.02\\
BCor-SIS & 0.8 & 0.80 & 0.150 & 0.122 & 0.131 & 1.00 & 0.92 & 0.10 & 0.00 & 0.00\\
BCodifCor-SIS & 0.8 & 0.80 & 0.212 & 0.163 & 0.163 & 1.00 & 0.92 & 0.02 & 0.00 & 0.00\\
BCor-SIS & 0.5 & 0.80 & 0.146 & 0.117 & 0.130 & 1.00 & 0.84 & 0.04 & 0.00 & 0.00\\
BCodifCor-SIS & 0.5 & 0.80 & 0.194 & 0.159 & 0.159 & 1.00 & 0.84 & 0.04 & 0.00 & 0.00\\
\addlinespace
BCor-SIS & 0.9 & 0.50 & 0.061 & 0.081 & 0.104 & 0.82 & 0.10 & 0.00 & 0.00 & 0.00\\
BCodifCor-SIS & 0.9 & 0.50 & 0.081 & 0.091 & 0.107 & 0.60 & 0.02 & 0.00 & 0.00 & 0.00\\
BCor-SIS & 0.8 & 0.50 & 0.055 & 0.074 & 0.096 & 0.88 & 0.06 & 0.00 & 0.00 & 0.00\\
BCodifCor-SIS & 0.8 & 0.50 & 0.076 & 0.083 & 0.107 & 0.62 & 0.02 & 0.00 & 0.00 & 0.00\\
BCor-SIS & 0.5 & 0.50 & 0.062 & 0.073 & 0.094 & 0.62 & 0.06 & 0.00 & 0.00 & 0.00\\
BCodifCor-SIS & 0.5 & 0.50 & 0.072 & 0.084 & 0.104 & 0.56 & 0.00 & 0.00 & 0.00 & 0.00\\
\bottomrule
\end{tabular}}
\end{table}

\begin{table}[!t]
\centering
\caption{BCodifCor-SIS and BCor-SIS performance for Gaussian random vectors.}
\label{tab:normal}
\rowcolors{2}{gray!15}{white}
\resizebox{\textwidth}{!}{%
\begin{tabular}{llrrrrrrrr}
\toprule
Method & $\rho$ & \multicolumn{3}{c}{$p_a$} & \multicolumn{5}{c}{$p_m$}\\
\cmidrule(lr){3-5}\cmidrule(lr){6-10}
 & & $d_1$ & $d_2$ & $d_3$ & $X_1$ & $X_2$ & $X_5$ & $X_{10}$ & $X_{15}$\\
\midrule
BCor-SIS & 0.95 & 0.738 & 0.464 & 0.348 & 1.00 & 1.00 & 0.94 & 0.80 & 0.40\\
BCodifCor-SIS & 0.95 & 0.824 & 0.558 & 0.423 & 1.00 & 1.00 & 0.90 & 0.68 & 0.30\\
BCor-SIS & 0.90 & 0.406 & 0.249 & 0.213 & 1.00 & 1.00 & 0.82 & 0.40 & 0.04\\
BCodifCor-SIS & 0.90 & 0.478 & 0.297 & 0.251 & 1.00 & 1.00 & 0.80 & 0.20 & 0.02\\
BCor-SIS & 0.80 & 0.202 & 0.144 & 0.139 & 1.00 & 1.00 & 0.48 & 0.00 & 0.00\\
BCodifCor-SIS & 0.80 & 0.232 & 0.161 & 0.157 & 1.00 & 1.00 & 0.24 & 0.00 & 0.00\\
BCor-SIS & 0.50 & 0.074 & 0.078 & 0.107 & 0.96 & 0.26 & 0.00 & 0.00 & 0.00\\
BCodifCor-SIS & 0.50 & 0.095 & 0.092 & 0.110 & 0.90 & 0.08 & 0.00 & 0.00 & 0.00\\
BCor-SIS & 0.20 & 0.043 & 0.065 & 0.091 & 0.12 & 0.00 & 0.00 & 0.00 & 0.00\\
BCodifCor-SIS & 0.20 & 0.048 & 0.070 & 0.092 & 0.06 & 0.00 & 0.00 & 0.00 & 0.02\\
\bottomrule
\end{tabular}}
\end{table}
\FloatBarrier

As illustrated in Tables \ref{tab:subgaussian} and \ref{tab:normal} 
for both simulated datasets from normal distribution and heavy-tailed distribution, Ball-codifference dominates Ball-Covariance in estimating the \(p_a\)s at all three levels of $d_1, d_2, d_3$, and for all values of $\rho$ and $\alpha$. 
Specifically, in Table \ref{tab:subgaussian}, the maximum gain in \(p_a\) attained by BCodifCor\-SIS over BCor\-SIS, ranges from $0.11$ (for $d_3$ at $\alpha=0.8$ and "$\rho=0.95$) to $0.16$ (for $d_1$ at $\alpha=0.9$ and $\rho=0.95$). converting this differences into percentages the (dividing by the values for BCor\-SIS) the maximum gain rates range form $26\%$ to $37\%$.

This fact confirms that Ball-codifference retains the highly correlated variables with teh outcome with a higher probability compared to Ball-Covariance. In a statistical prediction modeling this better performance can translate into a significantly higher prediction accuracy.

Overall, the codifference-weighted score improves the average retention rate \(p_a\) in most settings.  The gains are largest when the Toeplitz dependence is strong and the tails are heavy.  Exact-rank recovery \(p_m(j)\) is more variable because it requires the full ranking to place each selected neighbor at one specified position.  Thus, a method may improve top-\(d\) screening even when exact-rank probabilities at particular indices are similar or smaller.

\section{Data Example}\label{Data_Example}
To compare the prediction power between the BCor-SIS and BCodifCor-SIS, we applied both screening methods to the Riboflavin dataset [cite], a widely used benchmark for high-dimensional regression and feature screening. The dataset consists of gene expression measurements from 71 strains of Bacillus subtilis to predict riboflavin (vitamin B2) production. Each observation contains expression levels for 4,088 genes. The response variable is the logarithm of riboflavin production, which is continuous. Before running a linear regression, we prescreen the variables using both BCor-SIS and BCodifCor-SIS methods, for the number of covariates in 2-30,
and compare the prediction accuracy between the two models using the mean squared prediction error (MSPE). For each predetermined number of covariates, we divided the data into train (80\%) and test (20\%) and calculated MEPE based on $m=500$ iterations. \\ 
Table \ref{Riboflavin} illustrates that the BCodifCor-SIS screening method dominates BCor-SIS significantly, except for the two instances ($n_{cov}=6,7$) of small better performance for BCor-SIS. Specifically, the percentage of improvement in MSPE reaches even 35.6, 35.7, and 35.8 with $n_{cov}=11, 17, 23$. To account for the sensitivity of the results, we repeated the experiment with a different number of iterations ($m=100, 200$) and different proportions for dividing the data into train and test (70\% vs. 30\%  and 90\% vs. 10\%), and obtained similar results.   
%$n_{\mathrm{cov}}$ is the number of pre-screened covariates; RMSE$_{\mathrm{BCorr}}$ and  RMSE$_{\mathrm{BCodif}}$ are prediction means square error, 
%Diff$_{\mathrm{mean}}$ is the mean difference between them,
%$p$-value is P\-value for significance difference and Improvement (\%) is the percentage improvement in 

\begin{table}[!t]
\centering
\caption{Comparison of prediction accuracy of linear regression between BCodifCor-SIS and BCor-SIS pre-screening methods. $n_{\mathrm{cov}}$ is the number of pre-screened covariates; RMSE$_{\mathrm{BCorr}}$ and  RMSE$_{\mathrm{BCodif}}$ are prediction means square errors based on; Diff$_{\mathrm{mean}}=$RMSE$_{\mathrm{BCodif}}-$ RMSE$_{\mathrm{BCorr}}$,
$p$-value is the P-value for the statistical test of mean difference, and Improvement (\%) is the percentage improvement in prediction accuracy between the two methods.}
\label{Riboflavin}
\begin{tabular}{rrrrrr}
\toprule
$n_{\mathrm{cov}}$ & RMSE$_{\mathrm{BCorr}}$ & RMSE$_{\mathrm{BCodif}}$ &
Diff$_{\mathrm{mean}}$ & $p$-value & Improvement (\%) \\
\midrule
2  & 0.7322 & 0.7514 &  0.0192 & $<0.001$ &  2.6 \\
3  & 0.7641 & 0.7780 &  0.0139 & $<0.001$ &  1.8 \\
4  & 0.7650 & 0.7716 &  0.0066 & 0.0827   &  0.9 \\
5  & 0.7744 & 0.7877 &  0.0133 & 0.0039   &  1.7 \\
6  & 0.7885 & 0.7564 & -0.0320 & $<0.001$ & -4.1 \\
7  & 0.7816 & 0.7795 & -0.0021 & 0.6960   & -0.3 \\
8  & 0.7471 & 0.7998 &  0.0528 & $<0.001$ &  7.1 \\
9  & 0.7423 & 0.8257 &  0.0835 & $<0.001$ & 11.2 \\
10 & 0.7565 & 0.8365 &  0.0800 & $<0.001$ & 10.6 \\
11 & 0.6326 & 0.8577 &  0.2250 & $<0.001$ & 35.6 \\
12 & 0.6218 & 0.7688 &  0.1470 & $<0.001$ & 23.6 \\
13 & 0.6307 & 0.8006 &  0.1699 & $<0.001$ & 26.9 \\
14 & 0.6163 & 0.7677 &  0.1513 & $<0.001$ & 24.6 \\
15 & 0.6143 & 0.7934 &  0.1791 & $<0.001$ & 29.2 \\
16 & 0.5934 & 0.7817 &  0.1883 & $<0.001$ & 31.7 \\
17 & 0.5968 & 0.8096 &  0.2128 & $<0.001$ & 35.7 \\
18 & 0.6237 & 0.8117 &  0.1880 & $<0.001$ & 30.1 \\
19 & 0.6005 & 0.8408 &  0.2404 & $<0.001$ & 40.0 \\
20 & 0.6129 & 0.8054 &  0.1925 & $<0.001$ & 31.4 \\
21 & 0.6270 & 0.8264 &  0.1994 & $<0.001$ & 31.8 \\
22 & 0.6438 & 0.8368 &  0.1930 & $<0.001$ & 30.0 \\
23 & 0.6262 & 0.8504 &  0.2242 & $<0.001$ & 35.8 \\
24 & 0.6418 & 0.8154 &  0.1736 & $<0.001$ & 27.1 \\
25 & 0.6560 & 0.8337 &  0.1776 & $<0.001$ & 27.1 \\
26 & 0.6783 & 0.7448 &  0.0665 & $<0.001$ &  9.8 \\
27 & 0.7098 & 0.7572 &  0.0474 & $<0.001$ &  6.7 \\
28 & 0.7135 & 0.7991 &  0.0857 & $<0.001$ & 12.0 \\
29 & 0.6809 & 0.8202 &  0.1393 & $<0.001$ & 20.5 \\
\bottomrule
\end{tabular}
\label{tab:rmse_comparison}
\end{table}

\section{Discussion}\label{sec:discussion}

Ball-codifference screening is intended for high-dimensional problems in which ordinary covariance is either undefined or too unstable to serve as a reliable screening utility.  The statistic uses only indicator functions of random balls and bounded trigonometric contrasts.  Consequently, the construction does not require finite moments, while still retaining the geometric, model-free nature of Ball covariance.  This is useful in heavy-tailed applications where a small number of extreme observations may contain genuine signal rather than simple contamination.

The numerical results show that the proposed codifference weighting can improve top-ranked retention in the Toeplitz designs considered here.  The improvement is most visible when the dependence is strong and the marginal distribution is heavy-tailed.  The exact-rank probabilities are less stable, which is expected because they are more stringent than top-\(d\) recovery and are sensitive to small changes in neighboring marginal utilities.

Several practical points are worth noting.  First, the signed version of Ball-codifference is natural when the direction of the codifference contrast is meaningful, whereas the absolute value version may be preferable when the goal is purely to rank strength of association.  Second, the normalized statistic is scale-free but computationally more expensive because it requires marginal self-normalizers.  For very large \(p\), a computationally efficient workflow is to use the unnormalized score for an initial pass and then recompute the normalized score on a smaller candidate set.  Third, threshold selection can be handled as in other SIS procedures: either by retaining \(d_n=\lfloor n/\log n\rfloor\), or by using a data-adaptive cutoff calibrated by resampling.

The theory in this paper is stated under regularity conditions that make the random-ball empirical process and the local codifference denominators well behaved.  These conditions are standard for V-statistic and empirical-process arguments, but sharper finite-sample concentration bounds would be useful for automatic threshold selection.  Another useful direction is to develop faster implementations for grouped predictors and categorical responses without sacrificing the moment-free nature of the method.

\appendix

\section{Proof of Lemma~\ref{lem:extended_independence}}\label{app:proof_extended}

Let \((X,Y)\) be jointly symmetric \(\alpha\)-stable with \(1<\alpha<2\), and let \(\Gamma\) be its spectral measure on \(\Sph_2\).  For any real numbers \(a\) and \(b\), the characteristic function representation gives
\begin{equation}\label{eq:linear_stable_norm}
\|aX+bY\|_\alpha^\alpha
=\int_{\Sph_2}|as_1+bs_2|^\alpha\,\Gamma(ds).
\end{equation}
Substituting \((a,b)=(1,0),(0,1),(1,-1)\), and \((1,1)\) into \eqref{eq:linear_stable_norm} yields
\begin{align*}
\tau^*_{X,Y}
&=\int_{\Sph_2}\left\{|s_1|^\alpha+|s_2|^\alpha
-\frac{1}{2}\left(|s_1-s_2|^\alpha+|s_1+s_2|^\alpha\right)\right\}\Gamma(ds)  \\
&=\int_{\Sph_2} q_\alpha(s_1,s_2)\,\Gamma(ds),
\end{align*}
where
\[
q_\alpha(u,v)=|u|^\alpha+|v|^\alpha
-\frac{1}{2}\{|u-v|^\alpha+|u+v|^\alpha\}.
\]
We first record why the integrand is nonnegative.  Put \(p=\alpha/2\in(1/2,1)\).  By concavity of \(x\mapsto x^p\) on \([0,\infty)\),
\[
\frac{|u+v|^\alpha+|u-v|^\alpha}{2}
=\frac{\{(u+v)^2\}^{p}+\{(u-v)^2\}^{p}}{2}
\leq\left\{\frac{(u+v)^2+(u-v)^2}{2}\right\}^{p}
=(u^2+v^2)^p .
\]
Since \(p<1\), subadditivity of \(x^p\) gives
\[
(u^2+v^2)^p\leq |u|^{2p}+|v|^{2p}=|u|^\alpha+|v|^\alpha.
\]
Thus \(q_\alpha(u,v)\ge0\).  Moreover, the second inequality is strict whenever both \(|u|\) and \(|v|\) are positive.  Therefore \(q_\alpha(u,v)=0\) can occur only when \(uv=0\).  Conversely, if \(u=0\) or \(v=0\), then the equality is immediate.  Hence
\begin{equation}\label{eq:q_zero_axes}
q_\alpha(u,v)=0 \quad\Longleftrightarrow\quad uv=0.
\end{equation}

If \(X\) and \(Y\) are independent, the spectral measure of \((X,Y)\) is concentrated on the two coordinate axes.  Equation \eqref{eq:q_zero_axes} then gives \(q_\alpha=0\) \(\Gamma\)-almost everywhere, and consequently \(\tau^*_{X,Y}=0\).

Conversely, suppose \(\tau^*_{X,Y}=0\).  Since the integrand is nonnegative, \eqref{eq:q_zero_axes} implies that \(\Gamma\) is concentrated on
\[
\{(s_1,s_2)\in\Sph_2:s_1s_2=0\},
\]
the union of the two coordinate axes.  For arbitrary \(t,u\in\R\), the joint characteristic function then factorizes:
\begin{align*}
\E e^{i(tX+uY)}
&=\exp\left\{-\int_{\Sph_2}|ts_1+us_2|^\alpha\,\Gamma(ds)\right\}  \\
&=\exp\left\{-\int_{\{s_2=0\}}|ts_1|^\alpha\,\Gamma(ds)
      -\int_{\{s_1=0\}}|us_2|^\alpha\,\Gamma(ds)\right\} \\
&=\E e^{itX}\,\E e^{iuY}.
\end{align*}
Factorization of the joint characteristic function for all \((t,u)\) is equivalent to independence.  This proves the lemma.

\section{Proof of Theorem~\ref{thm:asymptotic_normality}}\label{app:proof_clt}

Let \(W=(X,Y)\) and let \(F\) denote the joint law of \(W\).  For \(w_r=(x_r,y_r)\), write
\[
A(w_1,w_2)=B_X(x_1,x_2)\times B_Y(y_1,y_2).
\]
For any probability measure \(G\) in a neighborhood of \(F\), define
\begin{align*}
p_G(w_1,w_2)&=G\{A(w_1,w_2)\},\\
p^X_G(w_1,w_2)&=G_X\{B_X(x_1,x_2)\},\\
p^Y_G(w_1,w_2)&=G_Y\{B_Y(y_1,y_2)\},\\
a_{1,G}(w_1,w_2)&=\int_{A(w_1,w_2)}\cos(x-y)\,G(dw),\\
a_{2,G}(w_1,w_2)&=\int_{A(w_1,w_2)}\cos x\,G(dw),\\
a_{3,G}(w_1,w_2)&=\int_{A(w_1,w_2)}\cos y\,G(dw).
\end{align*}
When the denominator is positive, set
\begin{equation}\label{eq:kappa_G_app}
\kappa_G(w_1,w_2)=\frac{a_{1,G}(w_1,w_2)}{p_G(w_1,w_2)}
-\frac{a_{2,G}(w_1,w_2)a_{3,G}(w_1,w_2)}{p_G^2(w_1,w_2)}.
\end{equation}
If trimming is used, the denominators \(p_G\) and \(p_G^2\) in \eqref{eq:kappa_G_app} are replaced by \(p_G\vee a_n\) and \((p_G\vee a_n)^2\).  The population functional can be written as
\begin{equation}\label{eq:Phi_functional}
\Phi(G)=\iint \left\{p_G(w_1,w_2)-p^X_G(w_1,w_2)p^Y_G(w_1,w_2)\right\}^2
\kappa_G(w_1,w_2)\,G(dw_1)G(dw_2).
\end{equation}
The empirical statistic \(\widehat\BCodif_s^{\,2}(X,Y)\) is exactly \(\Phi(F_n)\), where \(F_n\) is the empirical distribution, with the convention used in the statistic for empty empirical balls.  The population target is \(\Phi(F)=\BCodif_s^2(X,Y)\).

The proof has three steps.

\emph{Step 1: boundedness and smoothness of the kernel.}  The functions \(1_A\), \(\cos(x-y)1_A\), \(\cos x\,1_A\), and \(\cos y\,1_A\) are uniformly bounded by one.  Under Assumption~\ref{ass:main}, the class of product balls is Donsker and Glivenko-Cantelli; multiplying by a bounded measurable trigonometric function preserves these properties.  Hence the empirical process indexed by the finite collection of classes entering \eqref{eq:Phi_functional} is tight and asymptotically Gaussian.  On the event where \(p_G\) is bounded away from zero, the map
\[
(u,u_X,u_Y,v_1,v_2,v_3)
\mapsto (u-u_Xu_Y)^2\left(\frac{v_1}{u}-\frac{v_2v_3}{u^2}\right)
\]
is continuously differentiable with bounded first derivative.  In the trimmed case the same conclusion holds for the map with \(u\) replaced by \(u\vee a_n\), and Assumption~\ref{ass:main}(iii) ensures that the trimming error is asymptotically negligible.

\emph{Step 2: V-functional expansion.}  The preceding smoothness shows that \(\Phi\) is a composition of smooth finite-dimensional maps and V-functionals of the form
\[
G\mapsto \int\cdots\int h(w_1,\ldots,w_m)\,G(dw_1)\cdots G(dw_m)
\]
with bounded kernels.  Proposition~\ref{prop:vg_derivative}, applied componentwise and combined with the chain rule, yields quasi-Hadamard differentiability of \(\Phi\) at \(F\).  Therefore
\begin{equation}\label{eq:Phi_delta_app}
\sqrt n\{\Phi(F_n)-\Phi(F)\}=\dot\Phi_F\{\sqrt n(F_n-F)\}+o_p(1).
\end{equation}
Equivalently, after symmetrizing the bounded kernel, the same expansion is the first-order Hoeffding decomposition
\begin{equation}\label{eq:Hoeffding_app}
\sqrt n\{\Phi(F_n)-\Phi(F)\}
=\frac{m}{\sqrt n}\sum_{r=1}^n h_1(W_r)+o_p(1),
\end{equation}
where
\[
h_1(w)=\E\{h(w,W_2,\ldots,W_m)\}-\E\{h(W_1,\ldots,W_m)\}
\]
is the first projection of the symmetric kernel \(h\).  The order \(m\) is finite because every term in \eqref{eq:bcodif_empirical} is a finite product of empirical averages.  For example, after writing \(\widehat\kappa_{ij}\) as in \eqref{eq:local_codiff_empirical}, the highest-order term contains the two ball centers, two observations from the local codifference product, and four observations from \((\Delta^X_{ij}\Delta^Y_{ij})^2\).

\emph{Step 3: central limit theorem.}  The projection \(h_1(W)\) is bounded, hence has finite second moment.  Assumption~\ref{ass:main}(iv) excludes the degenerate case by requiring
\[
0<\sigma_{\BCodif}^2=m^2\var\{h_1(W)\}<\infty.
\]
The ordinary central limit theorem applied to \eqref{eq:Hoeffding_app} gives
\[
\sqrt n\left\{\widehat\BCodif_s^{\,2}(X,Y)-\BCodif_s^2(X,Y)\right\}
\dto N(0,\sigma_{\BCodif}^2).
\]
If the first projection is zero, the first-order term in \eqref{eq:Hoeffding_app} vanishes and the standard degenerate V-statistic limit applies instead.

\section{Proof of Theorem~\ref{thm:sure_screening}}\label{app:proof_sis}

Let \(E_n=\{\max_{1\le r\le p}|\widehat\omega_r-\omega_r|\le\epsilon_n\}\).  By assumption, \(\Prob(E_n)\to1\).  On \(E_n\), every active predictor \(r\in\calA\) satisfies
\[
\widehat\omega_r\ge \omega_r-\epsilon_n
\ge \min_{s\in\calA}\omega_s-\epsilon_n,
\]
whereas every inactive predictor \(r\in\calM\) satisfies
\[
\widehat\omega_r\le \omega_r+\epsilon_n
\le \max_{s\in\calM}\omega_s+\epsilon_n.
\]
The population separation condition gives
\[
\min_{s\in\calA}\omega_s-\epsilon_n
-\left(\max_{s\in\calM}\omega_s+\epsilon_n\right)
\ge 2c_n-2\epsilon_n.
\]
Since \(\epsilon_n=o(c_n)\), the right-hand side is positive for all sufficiently large \(n\).  Thus, on \(E_n\), every active variable has a larger empirical utility than every inactive variable.

For threshold screening, choose \(\eta_n\) such that
\[
\max_{s\in\calM}\omega_s+\epsilon_n < \eta_n < \min_{s\in\calA}\omega_s-\epsilon_n.
\]
Then all active variables are selected and all inactive variables are excluded on \(E_n\).  For top-\(d_n\) screening, if \(d_n\ge |\calA|\), the same ordering implies that the leading \(d_n\) empirical scores must contain \(\calA\).  Therefore
\[
\Prob(\calA\subseteq\widehat\calA_n)\ge \Prob(E_n)\to1,
\]
which proves the sure-screening property.

\section{Details for the local-denominator step}\label{app:local}

This section gives the algebra behind the denominator argument used in the asymptotic proof.  It is included because the empirical Ball-codifference statistic contains random local denominators, and those denominators must be handled with some care.

For fixed indices \((i,j)\), write
\[
D_{ij}=\Delta^{XY}_{ij}-\Delta^X_{ij}\Delta^Y_{ij},\qquad
C_{ij}^{(1)}=\frac{1}{n}\sum_{k=1}^n\delta^{XY}_{ij;k}\cos(X_k-Y_k),
\]
\[
C_{ij}^{(2)}=\frac{1}{n}\sum_{k=1}^n\delta^{XY}_{ij;k}\cos X_k,
\qquad
C_{ij}^{(3)}=\frac{1}{n}\sum_{k=1}^n\delta^{XY}_{ij;k}\cos Y_k.
\]
Then \eqref{eq:local_codiff_empirical} is
\[
\widehat\kappa_{ij}=\frac{C_{ij}^{(1)}}{\Delta^{XY}_{ij}}
-\frac{C_{ij}^{(2)}C_{ij}^{(3)}}{(\Delta^{XY}_{ij})^2}
\]
whenever \(\Delta^{XY}_{ij}>0\).  Hence
\begin{equation}\label{eq:bcodif_denominator_expanded}
\widehat\BCodif_s^{\,2}(X,Y)=\frac{1}{n^2}\sum_{i,j=1}^n
D_{ij}^2\left\{\frac{C_{ij}^{(1)}}{\Delta^{XY}_{ij}}
-\frac{C_{ij}^{(2)}C_{ij}^{(3)}}{(\Delta^{XY}_{ij})^2}\right\}.
\end{equation}
Expanding \(D_{ij}^2=(\Delta^{XY}_{ij})^2-2\Delta^{XY}_{ij}\Delta^X_{ij}\Delta^Y_{ij}+(\Delta^X_{ij}\Delta^Y_{ij})^2\) gives the more revealing form
\begin{align}\label{eq:bcodif_xi_form}
\widehat\BCodif_s^{\,2}(X,Y)
&=\frac{1}{n^2}\sum_{i,j=1}^n
\left(C_{ij}^{(1)}-\frac{C_{ij}^{(2)}C_{ij}^{(3)}}{\Delta^{XY}_{ij}}\right)
\Delta^{XY}_{ij}
\left(1-\frac{\Delta^X_{ij}\Delta^Y_{ij}}{\Delta^{XY}_{ij}}\right)^2  \\
&=\frac{1}{n^4}\sum_{i,j,l,k=1}^n
\{U_1(i,j,l,k)-U_2(i,j,l,k)\}
\left(1-2\xi_{ij}+\xi_{ij}^2\right),
\end{align}
where
\[
\xi_{ij}=\frac{\Delta^X_{ij}\Delta^Y_{ij}}{\Delta^{XY}_{ij}},
\]
\[
U_1(i,j,l,k)=\delta^{XY}_{ij;l}\delta^{XY}_{ij;k}\cos(X_k-Y_k),
\qquad
U_2(i,j,l,k)=\delta^{XY}_{ij;l}\cos X_l\,\delta^{XY}_{ij;k}\cos Y_k.
\]
The terms \(\xi_{ij}\) and \(\xi_{ij}^2\) may themselves be written as empirical sums, so \eqref{eq:bcodif_xi_form} is a finite-order V-statistic apart from the inverse powers of \(\Delta^{XY}_{ij}\).

The correct population target of \(\Delta^{XY}_{ij}\) is local.  Conditional on \((W_i,W_j)\), define
\begin{equation}\label{eq:pi_ij}
\pi_{ij}=F\{A(W_i,W_j)\}=\Prob\{W\in B_X(X_i,X_j)\times B_Y(Y_i,Y_j)\mid W_i,W_j\}.
\end{equation}
Thus the denominator is not estimated by a single deterministic constant; it is estimated by the random ball probability \(\pi_{ij}\).  Let
\[
\Delta^{XY,-ij}_{ij}=\frac{1}{n-2}\sum_{k\ne i,j}\delta^{XY}_{ij;k}.
\]
Conditional on \((W_i,W_j)\), the summands in \(\Delta^{XY,-ij}_{ij}\) are iid Bernoulli variables with mean \(\pi_{ij}\).  Consequently,
\[
\E\left\{\left(\Delta^{XY,-ij}_{ij}-\pi_{ij}\right)^2\mid W_i,W_j\right\}
=\frac{\pi_{ij}(1-\pi_{ij})}{n-2}\le \frac{1}{4(n-2)}.
\]
The difference between \(\Delta^{XY}_{ij}\) and \(\Delta^{XY,-ij}_{ij}\) is at most \(2/n\).  Therefore, by Jensen's inequality,
\begin{equation}\label{eq:average_delta_rate}
\E\left[\frac{1}{n^2}\sum_{i,j=1}^n |\Delta^{XY}_{ij}-\pi_{ij}|\right]
\le \frac{1}{2\sqrt{n-2}}+\frac{2}{n}=O(n^{-1/2}).
\end{equation}
In particular,
\[
\frac{1}{n^2}\sum_{i,j=1}^n |\Delta^{XY}_{ij}-\pi_{ij}|=O_p(n^{-1/2}).
\]
A uniform version follows from the Glivenko-Cantelli property of the product-ball class in Assumption~\ref{ass:main}.

The inverse maps are locally Lipschitz away from zero.  For completeness we state the elementary bound used repeatedly.

\begin{lemma}\label{lem:inverse_continuity}
Let \(a\in\{1,2\}\) and \(\eta>0\).  If \(|x-\eta|<\eta/2\), then
\[
|x^{-a}-\eta^{-a}|\le C_{a,\eta}|x-\eta|,
\]
where \(C_{a,\eta}=a(\eta/2)^{-(a+1)}\).  Moreover, for any \(0<\beta<1\), by reducing the neighborhood if necessary,
\[
|x^{-a}-\eta^{-a}|\le |x-\eta|^\beta.
\]
\end{lemma}

\begin{proof}
The derivative of \(x\mapsto x^{-a}\) is \(-a x^{-(a+1)}\).  On \((\eta/2,3\eta/2)\), its absolute value is bounded by \(C_{a,\eta}\).  The mean-value theorem gives the first inequality.  Since \(0<\beta<1\), \(|x-\eta|\le |x-\eta|^\beta\) whenever \(|x-\eta|\le1\).  Taking the neighborhood small enough that \(C_{a,\eta}|x-\eta|^{1-\beta}\le1\) gives the second inequality.
\end{proof}

If \(\pi_{ij}\ge\pi_0>0\) almost surely, Lemma~\ref{lem:inverse_continuity} and \eqref{eq:average_delta_rate} imply
\[
\frac{1}{n^2}\sum_{i,j=1}^n
\left|\frac{1}{\Delta^{XY}_{ij}}-\frac{1}{\pi_{ij}}\right|=O_p(n^{-1/2}),
\qquad
\frac{1}{n^2}\sum_{i,j=1}^n
\left|\frac{1}{(\Delta^{XY}_{ij})^2}-\frac{1}{\pi_{ij}^2}\right|=O_p(n^{-1/2}),
\]
on the event \(\min_{i,j}\Delta^{XY}_{ij}\ge\pi_0/2\), whose probability tends to one under the uniform law of large numbers.  If small balls occur, the trimmed denominators \(\Delta^{XY}_{ij}\vee a_n\) and \(\pi_{ij}\vee a_n\) give the same conclusion under the rate condition in Assumption~\ref{ass:main}(iii).  This justifies replacing the empirical denominators by their local population counterparts inside the first-order V-statistic expansion.

\section{A quasi-Hadamard differentiability argument for V-functionals}\label{app:hadamard}

This appendix records the functional-delta-method calculation used in the proof of Theorem~\ref{thm:asymptotic_normality}.  Let \(F\) be a distribution function on \(\R^q\), and let
\begin{equation}\label{eq:Vg}
V_g(F)=\int\cdots\int g(x_1,\ldots,x_m)F(dx_1)\cdots F(dx_m),
\end{equation}
where \(g\) is bounded and measurable.  For \(i=1,\ldots,m\), define
\[
g_{i,F}(x)=\int\cdots\int g(x_1,\ldots,x_{i-1},x,x_{i+1},\ldots,x_m)
\prod_{\ell\ne i}F(dx_\ell).
\]

Let \(\mathbb D^q_{inc}\) denote the space of bounded cadlag distribution-type functions on \(\R^q\), equipped with the supremum norm over common continuity points, and let \(\mathbb C^q\) be the corresponding continuous tangent space.  Assume that each \(g_{i,F}\) is locally of bounded variation and that the boundary terms arising in integration by parts vanish.

\begin{assumption}\label{ass:vg}
The following hold: (i) the integrals in \eqref{eq:Vg} exist for \(F\) and for the empirical distribution \(F_n\); (ii) \(|g|\le M\) for some finite \(M\); (iii) each \(g_{i,F}\) is locally of bounded variation and \(\int |dg_{i,F}|<\infty\); (iv) the integration-by-parts boundary terms vanish; and (v) \(\sqrt n(F_n-F)\dto B^\circ\) in the tangent space.
\end{assumption}

\begin{proposition}\label{prop:vg_derivative}
Under Assumption~\ref{ass:vg}, the functional \(V_g\) is quasi-Hadamard differentiable at \(F\) tangentially to \(\mathbb C^q\), with derivative
\begin{equation}\label{eq:vg_derivative}
\dot V_{g,F}(v)=\sum_{i=1}^m(-1)^q\int v(x-)\,dg_{i,F}(x).
\end{equation}
Consequently,
\[
\sqrt n\{V_g(F_n)-V_g(F)\}\dto \dot V_{g,F}(B^\circ).
\]
If \(B^\circ\) is a centered Gaussian process with covariance function \(\Gamma\), then \(\dot V_{g,F}(B^\circ)\) is normal with variance
\[
\sum_{i,j=1}^m\int\int \Gamma(x,y)\,dg_{i,F}(x)\,dg_{j,F}(y),
\]
whenever the displayed integrals exist.
\end{proposition}

\begin{proof}
Let \(h_n\to0\), let \(v_n\to v\) in the tangent norm, and put \(F_n^*=F+h_nv_n\).  We must show that
\[
\frac{V_g(F_n^*)-V_g(F)}{h_n}\to \dot V_{g,F}(v).
\]
Expanding the product measure in \eqref{eq:Vg} gives
\begin{align}\label{eq:product_expansion}
\frac{V_g(F_n^*)-V_g(F)}{h_n}
&=\sum_{i=1}^m
\int\cdots\int g(x_1,\ldots,x_m)
\prod_{\ell<i}F(dx_\ell)\,v_n(dx_i)\prod_{\ell>i}F(dx_\ell) \\
&\quad +\sum_{r=2}^m h_n^{r-1} Q_{n,r},\nonumber
\end{align}
where \(Q_{n,r}\) is a finite sum of integrals in which \(r\) coordinates are integrated with respect to \(v_n\) and the remaining coordinates with respect to \(F\).  Because \(|g|\le M\) and the total variations of \(v_n\) are bounded along convergent tangent sequences, there is a constant \(C\), independent of \(n\), such that \(|Q_{n,r}|\le C\) for every \(r=2,\ldots,m\).  Hence the second term in \eqref{eq:product_expansion} is \(o(1)\).

For the first-order terms, Fubini's theorem gives
\[
\int\cdots\int g(x_1,\ldots,x_m)
\prod_{\ell<i}F(dx_\ell)\,v_n(dx_i)\prod_{\ell>i}F(dx_\ell)
=\int g_{i,F}(x)\,v_n(dx).
\]
We next show that \(\int g_{i,F}\,dv_n\to \int g_{i,F}\,dv\).  Since \(g_{i,F}\) has finite total variation on compact sets and \(v_n\to v\) uniformly at continuity points, the convergence follows first on rectangles by a standard partition argument.  The tails are controlled by \(\int|dg_{i,F}|<\infty\) and the bounded total variations of \(v_n\).  Therefore
\[
\int g_{i,F}(x)\,v_n(dx)\to \int g_{i,F}(x)\,v(dx).
\]
Finally, repeated one-dimensional integration by parts in the \(q\) coordinates, with the boundary terms equal to zero by Assumption~\ref{ass:vg}, yields
\[
\int g_{i,F}(x)\,v(dx)=(-1)^q\int v(x-)\,dg_{i,F}(x).
\]
Summing over \(i\) proves the derivative formula \eqref{eq:vg_derivative}.  The weak convergence statement follows from the modified functional delta method applied to \(\sqrt n(F_n-F)\dto B^\circ\).  The variance formula follows by linearity of \(\dot V_{g,F}\) and the covariance structure of the Gaussian process.
\end{proof}

\end{document}